\documentclass[draft]{amsart}

\usepackage{amssymb}
\usepackage{url}
\usepackage{amscd}
\usepackage{enumerate}
\usepackage[all]{xy}

\def \N{{\mathbb N}}

\def \R{{\mathbb R}}

\def \1{{\mathbb 1}}

\theoremstyle{plain}
\newtheorem{theorem}{Theorem}
\newtheorem{proposition}{Proposition}
\newtheorem{definition}{Definition}
 
\newtheorem{corollary}{Corollary}

\theoremstyle{remark}

\newtheorem{Exemps}{Examples}
\newtheorem{Exemp}{Example}

\title[Porosity in the space of  Hölder-functions. ]{Porosity in the space of  Hölder-functions.}
\author{Mohammed Bachir}
\date{12/05/2021}

\address{Laboratoire SAMM 4543, Universit\'e Paris 1 Panth\'eon-Sorbonne\\ Centre P.M.F. 90 rue Tolbiac\\
75634 Paris cedex 13\\
France}
\email{Mohammed.Bachir@univ-paris1.fr}
\begin{document}
\begin{abstract}
Let  $(X,d)$ be a bounded metric space with a base point $0_X$, $(Y,\|\cdot\|)$ be a Banach space and $\textnormal{Lip}_0^\alpha (X,Y)$ be the space of all $\alpha$-Hölder-functions that vanish at $0_X$, equipped with its natural norm ($0<\alpha\leq 1$). Let $0<\alpha < \beta \leq 1$. We prove that  $\textnormal{Lip}_0^\beta(X,Y)$  is a $\sigma$-porous subset of $\textnormal{Lip}_0^\alpha(X,Y)$, if (and only if) $\inf\lbrace d(x,x'): x,x'\in X; x\neq x'\rbrace=0$  (i.e. $d$ is non-uniformly discrete). A more general result will be given.
\end{abstract}
\maketitle

\noindent {\bf Keyword, phrase:} vector-valued Lipschitz and H$\ddot{\textnormal{o}}$lder-functions, vector-valued Linear operators, $\sigma$-porosity, barrier cone.
\section{Introduction}
The main  result of this note  is Theorem \ref{Lip}, which gives a condition for some class of subsets of Lipschitz functions  to be $\sigma$-porous subsets. The result in the abstract, as well as all the other results of this note, are just a very immediate consequence of this main result. However, the main motivations which led to the main theorem of this note, was precisely the result mentioned in the abstract.
\vskip5mm
 Given a metric space $(X, d)$ with a distinguished point $0_X$ (called a base point of $X$)  and a Banach space  $(Y,\|\cdot\|)$, we denote by $\textnormal{Lip}_0(X_d,Y)$ (or  by $\textnormal{Lip}_0(X,Y)$, if no ambiguity arises)  the Banach space of all Lipschitz functions from $X$ into $Y$ that vanish at the base point $0_X$, equipped with its natural norm defined by 
\begin{eqnarray*}
 \|f\|_L:=\sup \lbrace \frac{\|f(x)-f(x')\|}{d(x,x')}: x, x'\in X; x\neq x\rbrace, \forall f \in \textnormal{Lip}_0(X_d,Y).
\end{eqnarray*}
We denote simply $\textnormal{Lip}_0(X_d)$ or $\textnormal{Lip}_0(X)$, if $Y=\R$. The space $L(X,Y)$ denotes the space of all linear bounded operators from $X$ into $Y$. The space $X^*$ denotes the topological dual of $X$. Notice that the space $\textnormal{Lip}_0(X,Y)$ can be  isometrically itentified to $L(\mathcal{F}(X),Y)$ where $\mathcal{F}(X)$ is the free-Lipschitz space over $X$ introduced by Godefroy-Kalton in \cite{GK}. Let us recall the definition of $\sigma$-porosity.
\begin{definition}\label{prous}  Let $(F, d)$ be a metric space and $A$ be a subset of $F$. A set $A$ of $F$  is called porous if there is a $c\in (0,1)$ so that for every $x\in A$  there are $(y_n)\subset F$ with $y_n \to  x$ and so that $B(y_n, c d(y_n, x)) \cap A = \emptyset$ for every $n$ (We denote by $B(z, r)$ the closed ball with center $z$ and radius $r$). A set $A$ is called $\sigma$-porous if it can be represented as a union $A = \cup_{n=0}^{+\infty}A_n $ of countably many porous sets (the porosity constant $c_n$ may vary with $n$).
\end{definition}
Every $\sigma$-porous set is of first Baire category. Moreover, in $\R^n$, every $\sigma$-porous set is of Lebesque measure zero. However,  there does exist a non-$\sigma$-porous subset of $\R^n$ which is of the first category and of Lebesgue measure zero (for more informations about $\sigma$-porosity, we refer to \cite{Z} and \cite{LP}). 

\vskip5mm
\paragraph{\bf The property $(\mathcal{P})$} Let $(X,d)$ be a metric space and $Y$ be a Banach space. Let $F$ be a nonempty (closed) convex cone of $\textnormal{Lip}_0(X,Y)$. We say that $F$ satisfies property $(\mathcal{P})$ if there exists a positive constant $K_F>0$ depending only on $F$ such that: 
$$(\mathcal{P}) \hspace{3mm} \forall (x, x')\in X\times X, \exists p\in F : \|p\|_L\leq K_F \textnormal{ and } \|p(x)-p(x')\|=d(x,x').$$
This property is related to the Hahn-Banach theorem and norming sets.
\begin{Exemps} \label{exemple1} The property $(\mathcal{P})$ satisfied in the following cases: 

$(i)$  if $X$ is a normed space and $F$ contains the space $X^*.e:=\lbrace x \mapsto p(x).e : p\in X^*\rbrace$, where $e\in Y$ is a fixed point  such that $\|e\|=1$.

$(ii)$ if $(X,d)$ is a metric space and $F$ contains the functions $d_z: x\mapsto d(x,z).e$, for all $z\in X$, where $e\in Y$ is a fixed point is such that $\|e\|=1$. 

$(iii)$  In particular, the space $\textnormal{Lip}_0(X,Y)$ satisfies the property  $(\mathcal{P})$. If moreover,  $X$ is a normed space, then $L(X,Y)$ has the property  $(\mathcal{P})$ too.
\end{Exemps}

\begin{proof}

$(i)$ By the Hahn-Banach theorem, for all $x\in X$ there exists $x^*\in X^*$ such that $\|x^*\|=1$ and $x^*(x)=\|x\|$. Then, for each $x\in X$, we consider the continuous  linear map $p_x=x^*.e: X\to Y$ defined by $p_x(z)=x^*(z)e$ for all $z\in X$, and the property $(\mathcal{P})$ is satisfied.

$(ii)$ Immediat.

$(iii)$ This part follows from $(i)$ and $(ii)$ respectively.
\end{proof}

\section{The main result}
We are going to give the proof of the main result of this note.  Let $(X,d)$  be a metric space with a base point $0$ and $Y$ be a Banach space.  Let $F\subset \textnormal{Lip}_0(X,Y)$ and $\phi: X\times X\to \R^+$ be a positive function such that $\phi(x,x')=0$ if and only if $x=x'$. For each real number $s>0$, we denote:
$$\mathcal{N}_{\phi,s}(F):=\lbrace f\in F: \sup_{x,x'\in X; x\neq x'} \frac{\|f(x)-f(x')\|}{\phi(x,x')} \leq s \rbrace,$$
$$\mathcal{N}_{\phi}(F):=\lbrace f\in F: \sup_{x,x'\in X; x\neq x'} \frac{\|f(x)-f(x')\|}{\phi(x,x')} <+\infty \rbrace.$$
Notice that $\mathcal{N}_{\phi}(F)=\cup_{k\in \N}\mathcal{N}_{\phi,k}(F)$ and $\mathcal{N}_{\psi}(F)\subset \mathcal{N}_{\phi}(F)$ if $\psi \leq \phi$.

\begin{theorem} \label{Lip}  Let $F$ be a nonempty (closed) convex cone of $\textnormal{Lip}_0(X,Y)$ satisfying $(\mathcal{P})$. Let $\phi: X\times X\to \R^+$ be any positive function such that $\phi(x,x')=0$ if and only if $x=x'$. Suppose that $\inf\lbrace \frac{\phi(x,x')}{d(x,x')}: x,x'\in X; x\neq x'\rbrace=0$, then for every positive real number $s>0$, we have that $\mathcal{N}_{\phi,s}(F)$ is a porous subset of $(F,\|\cdot\|_L)$.  Consequently, the following assertions are equivalent.

$(1)$ $\mathcal{N}_{\phi}(F) \neq F$.

$(2)$ $\inf\lbrace \frac{\phi(x,x')}{d(x,x')}: x,x'\in X; x\neq x'\rbrace=0$. 

$(3)$ $\mathcal{N}_{\phi}(F)$ is a $\sigma$-porous  subset of $(F,\|\cdot\|_L)$.
\end{theorem}
\begin{proof} $(1) \Longrightarrow (2)$. Suppose that  $\alpha:= \inf\lbrace \frac{\phi(x,x')}{d(x,x')}: x,x'\in X; x\neq x'\rbrace >0$, then  $\phi(x,x')\geq \alpha d(x,x')$ for all $x,x'\in X$. It follows that for every $f \in F$, we have that  
\begin{eqnarray*}
\sup_{x,x'\in X; x\neq x'} \frac{\|f(x)-f(x')\|}{\phi(x,x')} \leq \|f\|_L\sup_{x,x'\in X; x\neq x'} \frac{d(x,x')}{\phi(x,x')}\leq \frac{\|f\|_L}{\alpha} <+\infty .
\end{eqnarray*}
Thus,  $\mathcal{N}_{\phi}(F) = F$. Part $(3) \Longrightarrow (1)$ is trivial. 

Let us prove that if $\inf\lbrace \frac{\phi(x,x')}{d(x,x')}: x,x'\in X; x\neq x'\rbrace=0$, then for every $s>0$, we have that $\mathcal{N}_{\phi,s}(F)$ is a porous subset of $(F,\|\cdot\|_L)$, this gives in particular  $(2) \Longrightarrow (3)$. Indeed, if $\inf\lbrace \frac{\phi(x,x')}{d(x,x')}: x,x'\in X; x\neq x'\rbrace=0$, then there exists a pair of sequences $(a_n), (b_n)\subset X$ such that $0< r_n:=\frac{\phi(a_n,b_n)}{d(a_n,b_n)}  \to 0$. By assumption, there exists $K_F>0$ and a sequence $(p_n) \subset F$ such that $\|p_n\|_L\leq K_F$ and $\|p_n(a_n)-p_n(b_n)\|=d(a_n,b_n)$, for all $n\in \N$.  Let $f \in \mathcal{N}_{\phi,s}(F)$, then we have that $\sup_{x,x'\in X; x\neq x'} \frac{\|f(x)-f(x')\|}{\phi(x,x')} \leq s$. It follows that
\begin{eqnarray*}
\frac{\|(f+\sqrt{r_n}p_n)(a_n) -(f+\sqrt{r_n}p_n)(b_n) \|}{\phi(a_n,b_n)} &\geq& \sqrt{r_n} \frac{\|p_n(a_n)-p_n(b_n)\|}{\phi(a_n,b_n)} - \frac{\|-(f(a_n)-f(b_n))\|}{\phi(a_n,b_n)}\\
                                                                          &\geq&   \sqrt{r_n}\frac{d(a_n,b_n)}{\phi(a_n,b_n)} -  \sup_{x,x'\in X; x\neq x'}\frac{\|f(x)-f(x')\|}{\phi(x,x')}\\
                                                                           &\geq& \frac{1}{\sqrt{r_n}} - s
\end{eqnarray*}
Since, $r_n \to 0$, when $n\to +\infty$, there exists a subsequence  $(r_{n_m})$ such that  
$$ \frac{1}{\sqrt{r_{n_m}}} > 4s, \hspace{2mm} \forall m\in \N.$$ We set $f_m=f+\sqrt{r_{n_m}} p_{n_m}\in F$, for all $m\in \N$. We have that $$\|f_m-f\|_L=\sqrt{r_{n_m}}\|p_{n_m}\|_L\leq K_F\sqrt{r_{n_m}}\to 0 \textnormal{ when } m\to +\infty.$$ Let us prove that $B(f_m,\frac{1}{2K_F}\|f_m-f\|_L)\subset F\setminus \mathcal{N}_{\phi,s}(F)$ for all $m\in \N$. Indeed, let $g\in B(f_m,\frac{1}{2}\|f_m-f\|_L)$, then we have using the above informations that
\begin{eqnarray*}
\frac{\|g(a_{n_m})-g(b_{n_m})\|}{\phi(a_{n_m},b_{n_m})}&\geq& \frac{\|f_m(a_{n_m})-f_m(b_{n_m})\|}{\phi(a_{n_m},b_{n_m})}- \frac{\|(f_m-g)(a_{n_m})-(f_m-g)(b_{n_m})\|}{\phi(a_{n_m},b_{n_m})}\\
              &\geq& (\frac{1}{\sqrt{r_{n_m}}} - s) -\|f_m-g\|_L \frac{d(a_{n_m},b_{n_m})}{\phi(a_{n_m},b_{n_m})}\\
              &\geq&  (\frac{1}{\sqrt{r_{n_m}}} - s) -\frac{1}{2K_F}\|f_m-f\|_L \frac{d(a_{n_m},b_{n_m})}{\phi(a_{n_m},b_{n_m})}\\
 &\geq& (\frac{1}{\sqrt{r_{n_m}}} - s) -\frac{1}{2}\sqrt{r_{n_m}}\frac{1}{r_{n_m}}\\
&=& \frac{1}{2\sqrt{r_{n_m}}}-s\\
&>& s.
\end{eqnarray*}
Thus, we have that $g\in F\setminus \mathcal{N}_{\phi,s}(F)$ and so that  $B(f_m,\frac{1}{2}\|f_m-f\|_L)\subset F\setminus \mathcal{N}_{\phi,s}(F)$ for all $m\in \N$. Thus, $\mathcal{N}_{\phi,s}(F)$ is porous in $F$  (with $c=\frac{1}{2K_F}$). It follows that $\mathcal{N}_{\phi}(F)=\cup_{k\in \N} \mathcal{N}_{\phi,k}(F)$ is $\sigma$-porous in $(F,\|\cdot\|_L)$.
\end{proof}

\subsection{Immediate consequences}
We deduce immediately the result mentioned in the abstract.
\begin{corollary} \label{Porosity1} Let  $X_1:=(X,d_1)$ and $X_2:=(X,d_2)$ be a set equipped with two metrics such that $d_1\leq d_2$ and let $(Y,\|\cdot\|)$ be a Banach space. Then,  $\textnormal{Lip}_0(X_1,Y)$ is a $\sigma$-porous subset of $\textnormal{Lip}_0(X_2,Y)$ if (and only if) $d_1$ and $d_2$  are not equivalent, if and only if  $\textnormal{Lip}_0(X_1,Y)\neq \textnormal{Lip}_0(X_2,Y)$.
\end{corollary}
\begin{proof} We use Theorem \ref{Lip} and part $(iii)$ of Exemple \ref{exemple1}  observing the following equality $\textnormal{Lip}_0(X_1,Y)=\mathcal{N}_{d_1}(\textnormal{Lip}_0(X_2,Y)).$
\end{proof}
Notice that if  $0<\alpha \leq 1$  and $d$ is a metric, so is $d^\alpha$, hence the above corollay applies to the space of $\alpha$-Hölder-functions that vanish at $0_X$ which is $\textnormal{Lip}_0^\alpha(X,Y):=\textnormal{Lip}_0(X_{d^\alpha},Y)$. Notice also that if $0<\alpha < \beta \leq 1$ and $d$ is bounded, then $\textnormal{Lip}_0^\beta(X,Y) \subset \textnormal{Lip}_0^\alpha(X,Y)$. The metrics $d^\alpha$ and  $d^\beta$ are not equivalent if and only if, $\inf\lbrace \frac{d^\beta(x,x')}{d^\alpha(x,x')}: x,x'\in X; x\neq x'\rbrace=0$, if and only if $\inf\lbrace d(x,x'): x,x'\in X; x\neq x'\rbrace=0$ (since $\beta >\alpha$). Thus, we get the result of the abstract.

\begin{corollary} \label{Porosity2} Let  $(X,d)$ be a bounded metric space with a base point $0_X$, $(Y,\|\cdot\|)$ be a Banach space and $0<\alpha < \beta \leq 1$. Then, $\textnormal{Lip}_0^\beta(X,Y)$ is a $\sigma$-porous subset of $\textnormal{Lip}_0^\alpha(X,Y)$, if and only if $\inf\lbrace d(x,x'): x,x'\in X; x\neq x'\rbrace=0$.
\end{corollary}

Similarly to the case of lipschitz spaces, we obtain the following results in the linear case.
\begin{corollary} \label{Porosity}  Let  $X_1:=(X,\|\cdot\|_1)$ and $X_2:=(X,\|\cdot\|_2)$ be a linear space equipped with two norms such that $\|\cdot\|_1\leq \|\cdot\|_2$ and let $(Y,\|\cdot\|)$ be a Banach space. Then,  $L(X_1,Y)$ is a $\sigma$-porous  subset of $L(X_2,Y)$ if and only if  $\|\cdot\|_1$ and $ \|\cdot\|_2$ are not equivalent if and only if $L(X_1,Y)\neq L(X_2,Y)$.
\end{corollary}
\begin{proof}  We use Theorem \ref{Lip} and part $(iii)$ of Exemple \ref{exemple1} after observing that  $L(X_1,Y)=\mathcal{N}_{\|\cdot\|_1}(L(X_2,Y)).$
\end{proof}
\begin{Exemp} Let $i: (l^1(\N), \|\cdot\|_1)\to (l^1(\N), \|\cdot\|_{\infty})$ be the continuous  identity map. Then the image of the adjoint $i^*$ of $i$ is a $\sigma$-porous  subset of $(l^{\infty}(\N), \|\cdot\|_{\infty})$.
\end{Exemp}
We give in the following corollary a connexion between the surjectivity of the adjoint $T^*$ of a one-to-one bounded linear operator $T$ and the non-$\sigma$-porosity of its image (see in this sprit, the open mapping theorem in \cite[Theorem 2.11]{R}).
\begin{proposition} \label{metric}  Let $(X,\|\cdot\|_X)$ and $(Z,\|\cdot\|_Z)$ be Banach spaces.  Let $T :  X\to Z$ be a one-to-one  bounded linear operator and $T^*$ its adjoint. Then, the following assertions are equivalent.

$(i)$ $T^*(Z^*)$ is not  a $\sigma$-porous subset of $X^*$.

$(ii)$ There exists $\alpha>0$ such that $\alpha \|x\|_X\leq \|T(x)\|_Z$ for all $x\in X$.

$(iii)$ $T^*$ is onto.
\end{proposition}
\begin{proof} Since $T :  X\to Z$ is a one-to-one  bounded linear operator, then, the following map define another norm on $X$:
\begin{eqnarray*}
\|x\|:=\frac{\|T(x)\|_Z}{\|T\|}\leq \|x\|_X, \hspace{3mm} \forall x\in X.
\end{eqnarray*}
Let us denote $X_1:=(X,\|\cdot\|)$. By Corollary \ref{Porosity}, applied with $Y=\R$, we have that $X_1^*$ is a $\sigma$-porous subset of $X^*$  if and only if $\|\cdot\|$ and $\|\cdot\|_X$ are not equivalente. Thus, if $(ii)$ is not satisfied  (that is, $\|\cdot\|$ and $\|\cdot\|_X$ are not equivalente) then, since $T^*(Z^*)\subset X_1^*$ we get that  $T^*(Z^*)$ is contained in a $\sigma$-porous subset of $X^*$. Hence, $(i) \Longrightarrow (ii)$ is proved. Now, suppose that $(ii)$ holds, it follows that $T(X)$ is closed in $Y$. Let $x^*\in X^*$ and define $\phi$ on $T(X)$ by $\phi(T(x)):=x^*(x)$ for all $x\in X$. Clearly $\phi$  is well defined (since $T$ is one-to-one) and linear continuous on $T(X)$. Thus, $\phi$ extends to a linear continuous functional $y^*\in Y^*$ and we have $T^*(y^*)=y^*\circ T=x^*$. Hence, $T^*$ is onto and $(ii) \Longrightarrow (iii)$ is proved. Part $(iii) \Longrightarrow  (i)$, is trivial.
\end{proof}

Let $(X,\|\cdot\|)$ be a normed space, and let $S$ be a nonempty subset of the dual space $X^*$. The set $S$ is called separating if: $x^*(x) = 0$ for all $x^* \in S$ implies that $x = 0$. It is called norming if the functional 
$$N_S(x) =\sup_{x^* \in S; x^*\neq 0} \frac{|x^*(x)|}{\|x^*\|},$$ is an equivalent norm on $X$ (see \cite{GK} for the use of this notion). 

\begin{proposition} Let $(X,\|\cdot\|)$ be a normed space. Every  separating subset $S\subset X^*$  which is not  a $\sigma$-porous subset of $X^*$, is norming. 
\end{proposition}
\begin{proof} It is clear that $N(x)\leq \|x\|$ for all $x\in X$. On the other hand, we have that
$$S \subset \mathcal{N}_{N_S}(X):=\lbrace x^*\in X^* : \sup_{N_S(x)=1} |x^*(x)| <+\infty \rbrace.$$
Since $S$  is not contained in a $\sigma$-porous subset of $X^*$, then  $\mathcal{N}_{\phi}(X)$ must be non-$\sigma$-porous, which implies  from Theorem \ref{Lip} that $\inf_{\|x\|=1}N_S(x)>0$. Hence $N_S$ is equivalent to $\|\cdot\|$.
\end{proof}
\subsection{Coarse Lipschitz function and Lipschitz-free space}
Given a metric space $(X, d)$ with a base point $0_X$, the free space  $\mathcal{F}(X)$ is constructed as follows: we first consider as pivot space the Banach space $(\textnormal{Lip}_0(X),\|\cdot\|_L)$ of real-valued Lispchitz functions vanishing at the base point.
Then each $x \in X$ is identified to a Dirac measure $\delta_x$ acting linearly on $\textnormal{Lip}_0(X)$ as evaluation. Then the
mapping
\begin{eqnarray*}
\delta_X : X &\to& \textnormal{Lip}_0(X)^*\\
             x &\mapsto& \delta_x
\end{eqnarray*}
that maps $x$ to $\delta_x$ is an isometric embedding. The Lipschitz-free space $\mathcal{F}(X)$ over $X$ is defined as the
closed linear span of $\delta(X)$ in $\textnormal{Lip}_0(X)$.
  Furthermore, the free space is a predual for $\textnormal{Lip}_0(X)$, meaning
that $\mathcal{F}(X)^*$ is isometrically isomorphic to $\textnormal{Lip}_0(X)$.  Let $(X,d)$ and $(Y,d')$ be two metric spaces, each one with a base point ($0_X$ and $0_Y$ , respectively) and $F : X \to Y$ a Lipschitz function such that $R(0_X) = 0_Y$ . Then, it is well known (see \cite[Lemma ~2.2]{GK}) that there exists a unique linear operator $\widehat{F} : \mathcal{F}(X) \to \mathcal{F}(Y)$ such that $\|F\|_L= \|\widehat{F}\|$ and $\delta_Y \circ F=\widehat{F}\circ \delta_X$. The adjoint of $\widehat{F}$, namely $\widehat{F}^*:  \textnormal{Lip}_0(Y)\to  \textnormal{Lip}_0(X)$, satisfies $\widehat{F}^*(f)=f\circ F$ for all $f\in \textnormal{Lip}_0(Y)$.

 A map $F: (X,d) \to (Y,d')$  is said to be a coarse Lipschitz, if there exist $\alpha, \beta>0$ such that 
$$\alpha d(x,x') \leq d'(F(x),F(x'))\leq \beta d(x,x'), \hspace{3mm} \forall x, x' \in X.$$
Combining Proposition \ref{metric} together with a similar proof, we obtain in the following proposition, a characterization of coarse Lipschitz maps.
\begin{proposition} Let $(X,d)$ and $(Y,d')$ be metric spaces with base points $0_X$ and $0_Y$ respectively and let $F: (X,d) \to (Y,d')$  be a one-to-one  Lipschitz map such that $F(0_X)=0_Y$. Then the following assertions are equivalent.

$(i)$ The image of $\widehat{F}^*$  is not  $\sigma$-porous in $\textnormal{Lip}_0(X)$.

$(ii)$ The map $F$ is coarse Lipschitz.

$(iii)$ The adjoint  $\widehat{F}^*$ is onto.

$(iv)$ The linear map $\widehat{F} : \mathcal{F}(X) \to \mathcal{F}(Y)$ is coarse Lipschitz.
\end{proposition}
\begin{proof}  Since, $F$ is one-to-one, we define the following metric on $X$
$$d_1(x,x'):=\frac{1}{L_F}d'(F(x),F(x'))\leq d(x,x'), \hspace{3mm} \forall x, x' \in X,$$
where $L_F$ denotes the constant of Lipschitz of $F$. Suppose that $F$ is not coarse Lipschitz, then the metric $d_1$ is not equivalent to the metric $d$. It follows, using Corollary \ref{Porosity1},  that $\textnormal{Lip}_0(X_1)$ is $\sigma$-porous subset of $\textnormal{Lip}_0(X)$, where $X_1=(X,d_1)$. Now, we observe that $\textnormal{Im}(\widehat{F}^*):=\lbrace f\circ F : f\in \textnormal{Lip}_0(Y)\rbrace\subset \textnormal{Lip}_0(X_1)$, which implies that $\textnormal{Im}(\widehat{F}^*)$ is a $\sigma$-porous subset of $(\textnormal{Lip}_0(X),\|\cdot\|_L)$. Thus, we proved that $(i) \Longrightarrow (ii)$. Let us prove that $(ii) \Longrightarrow (iii)$. Let $g\in \textnormal{Lip}_0(X)$, we need to show that there exists $f\in \textnormal{Lip}_0(Y)$ such that $g=f\circ F$. Indeed, define $\phi: F(X)\to \R$ by $\phi(F(x)):=g(x)$ for all $x\in X$. The map $\phi$ is well defined since $F$ is one-to-one. On the other hand, $\phi$ is Lipschitz on $F(X)$ since $F$ is coarse Lipschitz. Thus, $\phi$ extends to a Lipschitz function $f$ from $Y$ into $\R$ with the same constant of Lipschitz, by the inf-convolution formula ($L_\phi$ denotes the constant of Lipschitz of $\phi$ on $F(X)$): $\forall y \in Y$
$$f(y):=\inf \lbrace \phi(y')+L_\phi d(y,y'): y'\in F(X)\rbrace.$$
Hence, $f\in \textnormal{Lip}_0(Y)$ and $f\circ F(x)=\phi(F(x))=g(x)$ for all $x\in X$ and so $(ii) \Longrightarrow (iii)$ is proved. Part $(iii) \Longrightarrow (i)$ is trivial. Now, from Proposition\ref{metric}, we see $(iii) \iff (iv)$.
\end{proof}
\subsection{Application to the barrier cone and polar of sets} Let $X$ be a normed space and $K$ be a nonempty subset of $X$. The  barrier cone of $K$ is the subset $B(K)$ of the topological dual $X^*$ defined by
$$B(K)=\lbrace x^* \in X^*: \sup_{x\in K} x^*(x)<+\infty \rbrace.$$
The polar set of $K$ is a subset of the barrier cone of $K$ defined as follows:
$$K^{\circ}=\lbrace x^* \in X^*: \sup_{x\in K} x^*(x) \leq 1 \rbrace.$$
\vskip5mm
The study of barrier cones has interested several authors. It is shown in ~\cite[Theorem 3.1.1]{La} that for a closed convex subset $K$ of $X$, we have that $B(K)=X^*$ if and only if $K$ is bounded, on the other hand, $B(K)$ is dense in $X^*$ if and only if $K$ does not contain any halfline. In general, we know that $\overline{B(K)}\neq X^*$ (see example in  \cite{AET}). A study of the closure of the barrier of a closed convex set is given in \cite{AET}. As an immediat consequence of  main theorem, we obtain bellow that the barrier cone of some general class of  unbounded subsets  is a $\sigma$-pourous subset of the dual $X^*$. This shows that in general, the barrier cone may be a "very small" subset of $X^*$.
\vskip5mm
We define the classe $\Phi(X)$ of positive functions on $X$ (not necessarily continuous)  as follows: $\phi \in \Phi(X)$ if and only if,  $\phi : X\to \R^+$ and satisfies 

$(i)$ $\phi(\lambda x)=|\lambda|\phi(x)$ for all $x\in X$ and all $\lambda \in \R$

$(ii)$ $\phi(x)=0$ if and only if $x=0$
\vskip5mm
For every  $\phi \in \Phi(X)$, we denote  $S_\phi:=\lbrace x\in X: \phi(x)= 1 \rbrace $ and $C_\phi:=\lbrace x\in X: \phi(x)\leq 1 \rbrace $. Notice, that in general $C_\phi$ is not a convex set (resp. not closed), if we do not suppose that $\phi$ is a convex function (resp. a continuous function). It is easy to see that 
\begin{eqnarray} \label{E1}
C_\phi \textnormal{ is bounded }\iff  S_\phi \textnormal{ is bounded } \iff  \inf_{x\in X: \|x\|=1}\phi(x) >0.
\end{eqnarray}
 Thanks to the symmetry of $\phi\in \Phi(X)$ and the fact that $B(C_\phi)=B(S_\phi)$, we have using the notation of Theorem \ref{Lip}, that
\begin{eqnarray} \label{E2}
B(C_\phi)=\mathcal{N}_{\phi}(X^*).
\end{eqnarray}
The polar of  $S_\phi$ coincides with $\mathcal{N}_{\phi,1}(X^*)$ and we have 
\begin{eqnarray} \label{E3}
C_\phi^{\circ}\subset S_\phi^{\circ} =\mathcal{N}_{\phi,1}(X^*).
\end{eqnarray}

Now, using (\ref{E1}), (\ref{E2}) and (\ref{E3}) and applying Theorem \ref{Lip} to the spaces $F=X^*$ and $Y=\R$ (using Exemple \ref{exemple1}), we get directly the following informations about the size of the barrier cone, as well as the polar of  a set of the form $C_\phi$ in the dual space.
\begin{corollary} Let $X$ be a normed space and $\phi \in \Phi(X)$. If  $C_\phi$ is not bounded in $X$, then the polar $C_\phi^{\circ}$  is contained in a porous subset of  $X^*$. Moreover, the followin assertions are equivalent.

$(i)$ $B(C_\phi)\neq X^*$.

$(ii)$ $C_\phi$ is not bounded in $X$.

$(iii)$ $B(C_\phi)$ is a $\sigma$-porous subset of $X^*$.
\end{corollary}
We deduce that, if  $K$ is any nonempty subset of $X$  such that  $S_\phi \subset K$, for some $\phi \in \Phi(X)$ with $S_\phi$ not bounded, then the polar $K^{\circ}$ is contained in a porous subset of $X^*$ and the barrier cone $B(K)$ is contained in a $\sigma$-porous subset of $X^*$. Notice that if  $K$ is a closed absorbing disk in $X$, does not contain a non-trivial vector subspace and is a neighborhood of the origin in $X$, then the Minkowski functional $\phi_K$ of $K$ is a continuous  norm (with respect to the norm $\|\cdot\|$, but not equivalent to it, if $K$ is not bounded) hence $\phi_K \in \Phi(X)$ and we have that $K=C_{\phi_K}$.


\bibliographystyle{amsplain}

\end{document}